\documentclass[10pt]{article}
\usepackage{amssymb}
\usepackage{amsmath}
\usepackage{amsthm}
\usepackage{color}
\usepackage{comment}
\usepackage{geometry}		
\usepackage{graphicx}
\usepackage{xfrac}

\let\originalleft\left
\let\originalright\right
\renewcommand{\left}{\mathopen{}\mathclose\bgroup\originalleft}
\renewcommand{\right}{\aftergroup\egroup\originalright}

\usepackage{tabularx,booktabs}
\usepackage{caption}
\usepackage{subcaption}

\usepackage{hyperref}
\hypersetup{
    colorlinks=true,
    linkcolor=black,
    urlcolor=blue,
    citecolor=blue
}

\usepackage{tikz}
\usepackage{array}
\usepackage{booktabs}
\usepackage{epstopdf}

\usepackage{fancyhdr}

\usepackage{mathtools}
\usepackage{tensor}

\newtheorem{theorem}{Theorem}[section]

\newtheorem{prop}[theorem]{Proposition}

\theoremstyle{definition}
\newtheorem{definition}{Definition}[section]

\theoremstyle{remark}

\usepackage{authblk} 
\usepackage{orcidlink} 

\usetikzlibrary{arrows.meta, positioning}

\title{Disrupting the scammer lifecycle: A dynamically-consistent numerical analysis of a compartment model for scam-victim dynamics}

\author[1]{Y.O.~Tijani\orcidlink{0000-0002-9127-7173}}

\author[2,3]{I.~Ghosh\orcidlink{0000-0001-8078-0577}}

\author[1]{S.D.~Oloniiju\orcidlink{0000-0002-9794-8580}}

\author[4]{H.O.~Fatoyinbo\,\orcidlink{0000-0002-6036-2957}\thanks{Corresponding author: \href{mailto:hammed.fatoyinbo@aut.ac.nz}{hammed.fatoyinbo@aut.ac.nz}}}

\affil[1]{Department of Mathematics, Rhodes University, Makhanda, PO Box 94, Grahamstown 6140, South Africa}
\affil[2]{School of Mathematical and Computational Sciences, Massey University, Colombo Road, Palmerston North, 4410, New Zealand}
\affil[3]{School of Mathematics and Statistics, University College Dublin, Dublin 4 D04 V1W8, Ireland}
\affil[4]{Department of Mathematical Sciences, School of Engineering, Computer and Mathematical Sciences, Auckland University of Technology, Auckland, 1010, New Zealand}

\date{\today}

\begin{document}

\maketitle

\begin{abstract}
Online deception and financial scams represent a pervasive threat in the digital age, yet a quantitative analysis and understanding of their propagation is lacking. This study introduces a novel model based on the framework of epidemiological models to describe the interaction between scammers and their victims. We propose a five-compartment deterministic model ($S-V-R-A_s-R_s$) calibrated using longitudinal data in fraud reports from the Canadian Anti-Fraud Centre. The model's theoretical properties are established, including the non-negativity of the state variables and the stability threshold defined by the basic reproduction number ($\mathcal{R}_0$). A non-standard finite difference scheme is developed for the numerical simulations to ensure dynamical consistency between the continuous deterministic model and its discrete equivalent. A key finding of the model sensitivity analysis indicates that the proliferation of scams is overwhelmingly driven by the lifecycle of scammers, their recruitment, attrition, and arrest, rather than the susceptibility of the victim population. The results of this study provide strong quantitative evidence that the most effective control strategies are those that directly disrupt the scammers' population. Overall, this study provides a crucial model for designing and evaluating evidence-based policies to combat the scourge of cybercrime. 
\end{abstract}

\section{Introduction}

The digital revolution has transformed the fabric of modern society, bringing unprecedented benefits and opportunities for global connectivity. However, this increased reliance on technology has also created a fertile ground for malicious activities, with cybercrime and scams emerging as a pervasive and evolving threat. Among the most widespread and damaging of these threats are online scams, which employ tactics ranging from phishing and social engineering to large-scale fraudulent schemes. As the boundaries of cyberspace continue to expand, the landscape of crime has shifted, with scammers exploiting vulnerabilities in digital systems, manipulating human psychology, and perpetrating sophisticated forms of online deception. These illicit activities inflict staggering financial losses on individuals and organizations globally, causing significant psychological distress and eroding trust in digital ecosystems \cite{SCAM}. Recent statistics highlight the scale and impact of cybercrime and online scams. In the first half of 2022, approximately $236.1$ million ransomware attacks occurred worldwide. The financial cost of these incidents is significant, with data breaches estimated to cost businesses an average of $\$4.88$ million in 2024. One in two American internet users had their accounts breached in 2021. The widespread nature of data exposure is also evident, as nearly $1$ billion emails were compromised in a single year, impacting one in five internet users \cite{SCAM}. Scams or cybercrime fraud fall into various thematic categories, including phishing, identity fraud, extortion, counterfeit merchandise, and false billing \cite{canadafraud2025}. Phishing is one of the most pervasive cyber threats facing businesses and individuals \cite{SCAM}. The adaptive and evolving nature of these scams presents a formidable challenge to law enforcement and cybersecurity professionals, necessitating a deeper, more systematic understanding of their underlying dynamics to develop effective countermeasures. 

To this end, mathematical modeling offers a powerful analytical lens. Drawing parallels with the spread of infectious diseases, compartmental models from epidemiology provide a robust framework for conceptualizing the propagation of scams through a population, allowing for quantitative insights into how a scam might spread, persist, or decline. Despite a growing interest in modeling the dynamics of these scams, a significant gap exists in the literature regarding continuous mathematical models. A notable exception is the study by Syniavska et al.~\cite{Syniavska}, which proposed a model for counteracting e-banking fraud using differential equations based on predator-prey dynamics, highlighting a promising yet underexplored avenue of research. The result of Syniavska et al.~\cite{Syniavska} indicates that two outcomes, a saddle point and a line of stable fixed points, are unrealistic in practice. These scenarios rely on the strict condition that each successful fraud attack generates only one new attack. In reality, the propagation rate is much higher. Therefore, the model suggests that the most plausible outcomes are a stable node and a stable degenerate node. Chikore et al.~\cite{Chikore} used a deterministic framework to analyze the propagation dynamics and potential impact of scams (cybercrime fraud). A central finding of their study was that the degree of heterogeneity in the scammer population is positively correlated with the overall level of victimization.

A comprehensive understanding of cybercrime fraud requires analyzing its co-evolution with the spread of public awareness, as modeled through the lens of scam rumors. Belhdid et al.~\cite{Belhdid} used the epidemiological approach to account for social media scam rumors. The study applies optimal control theory to assess how effectively different strategies can prevent the spread of fraudulent misinformation across social media platforms. Through computational modeling, they analyze the advantages and limitations of their proposed interventions, focusing on their capacity to enhance the reliability of digital information. The study of Nwaibeh and Chikwendu~\cite{Nwaibeh} investigates the causative impact of governmental policies in mitigating the effect of scam rumor using the susceptible-exposed-ignorant-scammers-stiflers $(S-E-I-S_{c}-S_{t})$ model. Komi~\cite{Komi} incorporated education levels into their rumor propagation model by dividing uninformed individuals into two distinct categories: those with higher education and those without formal education. Their research demonstrated that educational background significantly influences how rumors circulate through populations. Additional studies in this direction can be found in the following articles and the references therein \cite{Zanette, Tadmon, Daley, Cheng}.

To contextualize this scam discussion within broader crime research, Catalayud et al.~\cite{Catalayud} developed a mathematical framework for modeling criminal behavior patterns using Spanish crime statistics as their primary dataset. Their compartmental model consisted of three population groups: $S_{t}$ representing the overall criminal population, $L_{t}$ denoting criminals currently at liberty, and $P_{t}$ indicating the incarcerated criminal population. Catalayud et al.~\cite{Catalayud} demonstrated that this continuous mathematical model functions effectively as a predictive tool for crime trends. Nuno et al~\cite{Nuno} discussed the vulnerability of a crime-prone society. The study compared two intervention strategies: police efficiency and criminal appeal rate reduction. Park and Kim~\cite{Park} proposes a dynamic model for the interactions between different offender types. The mathematical framework is compartmental, grouping individuals by crime severity and arrest status with assumed equal transition probabilities. The study explicitly incorporates the concepts of ``broken windows effect" and ``prison as a crime school". The analysis of the system identifies how parameters like arrest rate, sentence length, and inmate contact affect the equilibrium distribution of criminals. A key counter-intuitive finding is that more extended imprisonment, without controlled inmate contact, can increase serious crime. The study also addresses optimal police resource deployment for crime control. Mathematical models have been widely applied to analyze localized crime patterns, such as re-victimization and hot spot formation, often using reaction-diffusion equations \cite{Short, Short1}. Conceptually, these frameworks are drawn from population biology, borrowing from established infectious disease \cite{McMillon} and predator-prey models \cite{Vargo, Nuno1}. This epidemiological viewpoint has likewise been extended to organized crime, where gang membership is modeled as a contagious process \cite{Nyabadza, Sooknanan}. A substantial body of research has employed continuous models to explore the impact of crime and its connections to various social factors \cite{Sooknanan}.

The novelty of this study lies in its continuous model (epidemiological approach) formulation of scam dynamics, which contributes to the growing body of knowledge on cybercrime and scams by shedding light on the complex dynamics of these threats. To the author's knowledge, this work appears to be the first to approach this study using a continuous mathematical model. This study has implications for policymakers, law enforcement, and individuals seeking to protect themselves from cyber threats, highlighting the need for a multifaceted approach to combat these evolving threats in the digital age.

The remainder of this study is structured into four subsequent sections. Section~\ref{MF_sect} addresses the mathematical formulation of the scam (cybercrime) victim dynamics. Section~\ref{MA_sect} analyzes the qualitative properties of the continuous model, specifically investigating the positivity of the model's variables and the existence and stability of equilibrium points. In Section~\ref{NSFD_sect}, the estimation of the controlling parameter for the dynamical system, based on real-life data from the Canadian Anti-Fraud Centre, is detailed. We outline the numerical technique used to approximate the continuous differential equations and present the simulation results obtained using the non-standard finite difference method. The sensitivity analysis concerning the model parameters is also presented in Section~\ref{NSFD_sect}. Section~\ref{Conc_sect} serves as the conclusion, summarizing the key findings, highlighting the study's significance, and suggesting directions for future work. 

\section{Model formulation}\label{MF_sect}

\noindent In this study, we consider the dynamics of the scam victim through deterministic ordinary differential equations. The population ($N$) is divided into five subpopulations: susceptible ($S$), victims ($V$), recovered victims ($R$), active scammers ($A_s$), and removed scammers ($R_s$), in a way that at any time $t$,
$$N(t)=S(t)+ V(t)+R(t)+A_{s}(t)+R_{s}(t).$$





 
     

\begin{figure}[h!]
    \centering
\begin{tikzpicture}[node distance=2.cm, auto, >=Stealth, font = \large]

    \node [draw, circle, minimum size=1cm, line width= 2pt, fill=blue!50] (S) {S};
    \node [draw, circle, right=of S, minimum size=1cm, line width= 2pt, fill=yellow!50] (V) {V};
    \node [draw, circle, right=of V, minimum size=1cm, line width= 2pt, fill=green!50] (R) {R};

    \node [draw, circle, below=of S,xshift=25mm, minimum size=1cm, line width= 2pt, fill=magenta!50] (As) {A\(_s\)};
    \node [draw, circle, right=of As,minimum size=1cm, line width= 2pt, fill=lime!50] (Rs) {R\(_s\)};

    \draw [->, line width = 0.7mm] (S) -- (V) node[midway, above] {$\beta$};
    \draw [->, line width = 0.7mm] (V) -- (R) node[midway, above] {$\gamma$};
    \draw [->, line width = 0.7mm] (V) -- (As) node[midway, above, xshift=3mm] {$\psi$};
    \draw [->, line width = 0.7mm] (As) -- (Rs) node[midway, above] {$\lambda$};
     \path[->, line width = 0.7mm] (R) edge[bend left=290] node[above] {$\sigma$} (S);
 
\draw [->,dashed, line width = 0.7mm] (As) -- (S) node[midway, right,color=black] {};
\path[->, line width = 0.7mm] (0.5,-3.05) edge node[swap, yshift=5.5mm, xshift=1mm] {$\delta$} (As);  
 \path[->, line width = 0.7mm]  (As) edge node[swap, yshift=-5mm,xshift=3mm] {$\mu$} +(0,-1);
     
\end{tikzpicture}
\caption{Flow diagram of the scam victim model. Solid black arrows denote the recruitment and progression pathway. The dashed black arrow denotes transmission}
\label{fig:compartments}
\end{figure}
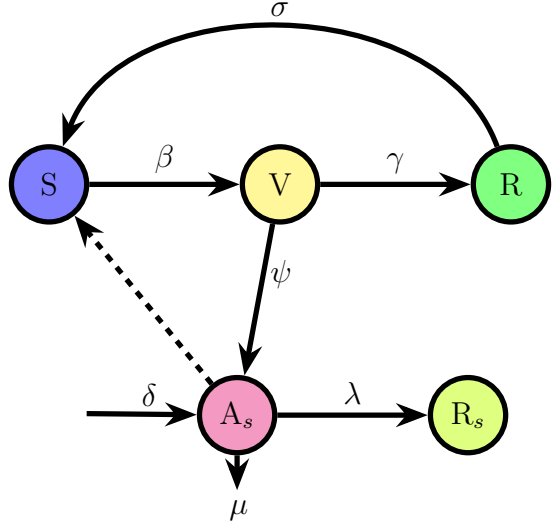

\noindent The scam pathway in the population, as described in the flow diagram in Figure~\ref{fig:compartments}, is as follows: Susceptible individuals become victims after an interaction with an active scammer. This occurs at a rate of $\beta$. The population grows further due to the possibility that the recovered individuals have been scammed, which occurs at a rate $\sigma$. Thus, the time evolution of the susceptible population is given as 
$$ \frac{dS}{dt}=-\beta S A_{s} + \sigma R.$$

The victim population arises as a result of interaction with an active scammer. Victims will leave the compartment in two ways, either as recovered individuals at a rate of $\gamma$ or as active scammers at a rate $\psi$. So, the following differential
equation describes how the victim population evolves over time
$$ \frac{dV}{dt}= \beta S A_{s} -\gamma V - \psi V.$$

The compartment for recovered individuals is populated due to awareness after the scam, this occurs at rate $\gamma$. In addition, it is assumed that the recovered individuals are not completely free of scammers; therefore, they become susceptible again at a rate $\sigma$. Thus, the recovered compartment is modeled by
$$\frac{dR}{dt}= \gamma V -\sigma R.$$

Active scammers are recruited at a rate of $\delta$ and removed in two ways: being caught by a law enforcement agent at a rate $\lambda$ or an active scammer chooses to quit the scamming activity at a rate of $\mu$. The dynamics of active scammers is given by
$$ \frac{dA_s}{dt}=\delta A_{s} -\mu A_{s} - \lambda A_{s} + \psi V. $$
Lastly, we represent the time evolution of the removed scammers by the following differential equation: 
$$\frac{dR_s}{dt}=\lambda A_{s}.$$

It follows that the scam victim model is governed by a system of ODEs as follows:
\begin{equation}
    \begin{aligned}
        \frac{dS}{dt}&=-\beta S A_{s} + \sigma R,\\
        \frac{dV}{dt}&= \beta S A_{s} -\gamma V - \psi V,\\
        \frac{dR}{dt}&= \gamma V -\sigma R,\\
        \frac{dA_s}{dt}&=\delta A_{s} -\mu A_{s} - \lambda A_{s} + \psi V, \\
        \frac{dR_s}{dt}&=\lambda A_{s},
    \end{aligned}
    \label{eq:model}
\end{equation}
with $$N=S+V+R+A_s+R_s.$$ The initial conditions at the initial time t = 0 are taken as $$S\ge 0, V\ge 0, R\ge 0, A_s \ge 0, R_s \ge 0.$$ The description of the model parameters of the scam victim model is presented in Table~\ref{tab:paramdescr}.

\begin{table}[h!]
\centering
\captionsetup{justification=centering,,margin=1cm}
\caption{Description of model parameters in \eqref{eq:model}}
 \begin{tabular}{ll}
	\hline
    \hline
	Parameter & Description  \\ \hline
	$\beta$	& Rate at which susceptible fall victim 	\\
   $\sigma$ & Rate at which recovered individuals become susceptible again \\
	$\gamma$	& Rate at which victims permanently learn and become resistant\\
    $ \psi$ & Rate at which some victims joined active scammers  \\
    $\delta$ & Scammer recruitment rate  \\
    $\mu$ &  Rate at which scammers quit  \\
    $\lambda$	& Rate at which scammers get caught by law enforcement agents \\ 
	\hline
    \hline
\end{tabular}
\label{tab:paramdescr}
\end{table}

\section{Qualitative analysis of the model}\label{MA_sect}
We now investigate the qualitative behavior of the model in \eqref{eq:model}. Specifically, we analyze the positivity and boundedness of the solutions, and identify the region in which the system dynamics is physically and mathematically well-defined. This forms a foundational step toward establishing the existence and stability of equilibria.

\subsection{Boundedness and positivity of solutions}
We begin by showing that the total population remains bounded for all time and that all state variables remain non-negative, provided the initial conditions are non-negative.

\begin{theorem}[Boundedness]
Let the total population be defined as $N(t) = S(t) + V(t) + R(t) + A_s(t) + R_s(t)$, and consider the region
\begin{align}
\label{eqn:Gamma}
\Gamma = \left\{(S, V, R, A_s, R_s) \in \mathbb{R}_+^5: 0 < N(t) \le N(0)\exp(\delta t) \right\}.
\end{align}
Then, for any non-negative initial condition $(S(0), V(0), R(0), A_s(0), R_s(0)) \in \mathbb{R}_+^5$, the solution of system~\eqref{eq:model} remains bounded in the region $\Gamma$. 
\end{theorem}

\begin{proof}
By summing all five equations in system~\eqref{eq:model}, we obtain the governing equation for the total population:
\begin{align}
\frac{dN}{dt} = (\delta - \mu) A_s \le \delta A_s \le \delta N. \label{eq:dNdt}
\end{align}
This inequality \eqref{eq:dNdt} is a linear differential inequality for $N(t)$. Its solution satisfies $N(t) \le N(0) \exp(\delta t)$, which shows that $N(t)$ is bounded above in the region $\Gamma$ for all $t \ge 0$.
\end{proof}

\begin{theorem}[Positivity]\label{thm:positivity}
Let the initial conditions satisfy $(S(0), V(0), R(0), A_s(0), R_s(0)) \in \mathbb{R}_+^5$. Then the solution to system~\eqref{eq:model} remains non-negative for all $t > 0$, and hence the state variables evolve within the region $\Gamma$.
\end{theorem}

\begin{proof}
Define the maximal interval of existence for which the solution remains non-negative as
\begin{align}
T = \sup \left\{t > 0 : S(t), V(t), R(t), A_s(t), R_s(t) \ge 0 \right\}. \label{eq:sup}
\end{align}
We proceed to verify that each component of the system remains non-negative for all $t \in [0, T]$ by deriving lower bounds:

\begin{enumerate}
    \item[(i)] From the differential equation for $S$, we have
    \begin{align}
    \frac{dS}{dt} \ge -\beta S A_s. \label{eq:S}
    \end{align}
    Applying the integrating factor method to \eqref{eq:S} , we get
    \begin{align}
    \frac{d}{dt} \left( S(t) \exp\left(\int_0^t \beta A_s(\tau) \, d\tau \right) \right) \ge 0.\label{eq:S2}
    \end{align}
    Integrating both sides of \eqref{eq:S2} from $0$ to $t$ yields
    \begin{align}
        S(t) \ge S(0) \exp\left(-\int_0^t \beta A_s(\tau) \, d\tau \right) \ge 0.
    \end{align}

    \item[(ii)] For $V$, we have
    \begin{align}
         \frac{dV}{dt} \ge -(\gamma + \psi)V,
    \end{align}
    which gives
    \begin{align}
        V(t) \ge V(0) \exp(-(\gamma + \psi)t) \ge 0.
    \end{align}

    \item[(iii)] For $R$, the equation implies
    \begin{align}
    R(t) \ge R(0) \exp(-\sigma t) \ge 0.
    \end{align}

    \item[(iv)] Similarly, the last two equations yield
    \begin{align}
    A_s(t) \ge A_s(0) \exp(-(\mu + \lambda)t) \ge 0, \quad R_s(t) \ge R_s(0) \ge 0.
    \end{align}
\end{enumerate}
Each state variable remains non-negative on the interval $[0, T]$. Hence, by continuity of solutions and the non-negativity of their derivatives at the boundary of the state space, the solution can be extended beyond $T$, contradicting the maximality of $T$ unless $T = \infty$. Therefore, the solutions remain non-negative for all $t > 0$.
\end{proof}

As an additional check on the dynamics near the boundaries of the positive orthant, we observe:
\begin{equation}
\left. 
\begin{aligned}
    &\left.\frac{dS}{dt}\right|_{S = 0} = \sigma R \ge 0, \quad
\left.\frac{dV}{dt}\right|_{V = 0} = \beta S A_s \ge 0, \quad
\left.\frac{dR}{dt}\right|_{R = 0} = \gamma V \ge 0,\\
&\left.\frac{dA_s}{dt}\right|_{A_s = 0} = \psi V \ge 0, \quad
\left.\frac{dR_s}{dt}\right|_{R_s = 0} = \lambda A_s \ge 0.
\end{aligned}
\qquad \quad \right\}
\end{equation}
These boundary conditions confirm that trajectories cannot cross out of the positive orthant, and the region $\Gamma$ is positively invariant.

\subsection{Scam-free equilibrium (SFE)}
We now analyze the scam-free equilibrium (SFE) of model~\eqref{eq:model}, denoted by
\begin{align}
\label{eq:sfe}
\phi_0 = (S^*, V^*, R^*, A_s^*, R_s^*) = (N, 0, 0, 0, 0).
\end{align}

To understand the stability of $\phi_0$, we compute the basic reproduction number, denoted by $R_0$. In epidemiological models, $R_0$ represents the expected number of new cases directly generated by one infectious individual in a fully susceptible population. In this context, it quantifies the expected number of individuals who get scammed due to the presence of a single scammer, hence referred to as the scam reproduction number. Two compartments are directly associated with individuals influenced by a scam, the $V$-class and the $A_s$-class. Let $y_i$ represent the population in the $i$-th scammed compartment, where $i = 1$ corresponds to $V$ and $i = 2$ to $A_s$. Then the dynamics of each $y_i$ class can be written in the form
\begin{align}
    \frac{dy_i}{dt} = F_i(y) - G_i(y),
\end{align}

where $F_i(y)$ denotes the rate at which new individuals who are scammed appear in compartment $i$, and $G_i(y)$ accounts for all other inflow and outflow dynamics of compartment $i$.

The next generation matrix method leads us to define the following matrices
\begin{align}
    F = \begin{bmatrix}
    0 & \beta S^* \\
    0 & 0
\end{bmatrix}, \quad \text{and} \quad 
V = \begin{bmatrix}
    \gamma + \psi & 0 \\
    -\psi & \mu + \lambda - \delta
\end{bmatrix},
\end{align}
where the next generation matrix is then defined as
\begin{align}
    K = FV^{-1} = \frac{1}{(\gamma + \psi)(\mu + \lambda - \delta)} 
\begin{bmatrix}
    \beta S^* \psi & \beta S^*(\gamma + \psi) \\
    0 & 0
\end{bmatrix}.
\end{align}

The basic reproduction number is defined as the spectral radius of $K$, which in this case gives
\begin{align}
\label{eq:R0}
\mathcal{R}_0 = \frac{\beta \psi N}{(\gamma + \psi)(\mu + \lambda - \delta)},
\end{align}
where we used $S^* = N$ at the scam-free equilibrium.

\begin{theorem}[Local stability of the SFE]
The scam-free equilibrium $\phi_0$ is locally asymptotically stable if $\mathcal{R}_0 < 1$, and unstable if $\mathcal{R}_0 > 1$.
\end{theorem}

\begin{proof}
We analyze the eigenvalues of the Jacobian matrix of the system of ODEs \eqref{eq:model} evaluated at $\phi_0$
\begin{align}
    J = \begin{bmatrix}
        0 & 0 & \sigma & - \beta N & 0 \\
        0 & -(\gamma + \psi) & 0 & \beta N & 0 \\
        0 & \gamma & -\sigma & 0 & 0 \\
        0 & \psi & 0 & \delta - \mu - \lambda & 0 \\
        0 & 0 & 0 & \lambda & 0
\end{bmatrix}.
\end{align}

The eigenvalues $\varepsilon_i, i =1,2,\ldots,5$ of the Jacobian matrix $J$ can be evaluated by solving the characteristics equation $|J-\varepsilon I| = 0$. The eigenvalues are determined as follows: one eigenvalue is $\varepsilon_3 = -\sigma < 0$, two eigenvalues $\varepsilon_{1} = \varepsilon_2 = 0$, and the last two eigenvalues, $\varepsilon_4$ and $\varepsilon_5$, come from the submatrix:
\begin{align}
    \begin{bmatrix}
    -(\gamma + \psi) & \beta N \\
    \psi & \delta - \mu - \lambda
\end{bmatrix}.
\end{align}
This yields the characteristic equation:
\begin{align}
    \varepsilon^2 - (p_1 + p_2)\varepsilon + (p_1p_2 - \beta N \psi) = 0,
\end{align}
where $p_1 = -(\gamma + \psi)$ and $p_2 = \delta - \mu - \lambda$. The necessary conditions for the roots of this characteristic equation to be real and negative are (i) positive discriminant, (ii) $-(p_1+p_2)>0$, and (iii) $p_1 p_2 - \beta N\psi > 0$. Now, we examine the signs of the roots. The discriminant is given by
\begin{align}
    \Delta = (p_1 + p_2)^2 - 4(p_1p_2 - \beta N \psi) = (p_1 - p_2)^2 + 4\beta N \psi > 0, 
\end{align}
So, the roots are real. The sum of the roots is negative since $-(p_1 + p_2) > 0$. The product of the roots is positive if and only if $p_1p_2 - \beta N \psi > 0$. Now, noting that
\begin{align*}
    \mathcal{R}_0 = \frac{\beta N \psi}{-p_1 p_2},
\end{align*}
we conclude that $\mathcal{R}_0 < 1$ is equivalent to $p_1p_2 - \beta N \psi > 0$. Therefore, all eigenvalues have negative real parts when $\mathcal{R}_0 < 1$, proving local stability of the scam-free equilibrium.
\end{proof}

\begin{theorem}[Global stability of the SFE]
The scam-free equilibrium $\phi_0$ is globally asymptotically stable in $\Gamma$ when $\mathcal{R}_0 < 1$.
\end{theorem}

\begin{proof}
To prove the global stability, we construct a Lyapunov function of the form
\begin{align}
    L = a_1 V + a_2 A_s,
\end{align}
with $a_1, a_2 > 0$ to be determined. Then,

\begin{align}
\frac{dL}{dt} &= a_1 \frac{dV}{dt} + a_2 \frac{dA_s}{dt} = a_1[\beta S A_s - (\gamma + \psi)V] + a_2[(\delta - \mu - \lambda)A_s + \psi V].
\end{align}

We eliminate $V$ terms by choosing $a_1 = \frac{a_2 \psi}{\gamma + \psi}$. Substituting, we get:
\begin{align}
    \frac{dL}{dt} = a_2 A_s (\gamma + \mu - \delta) \left[ \frac{\beta \psi S}{(\gamma + \psi)(\lambda + \mu - \delta)} - 1 \right].
\end{align}

Since $S \le N$, we have
\begin{align}
    \frac{dL}{dt} \le a_2 A_s (\gamma + \mu - \delta)(\mathcal{R} - 1).
\end{align}

Hence, if $\mathcal{R}_0 < 1$, then $\frac{dL}{dt} \le 0$, with equality only at the SFE. Therefore, by Lyapunov's direct method, $\phi_0$ is globally asymptotically stable in $\Gamma$, which means every trajectory reaches $\phi_0$ whenever $\mathcal{R}_0 < 1$.
\end{proof}

\subsection{Scam-endemic equilibrium (SEE)}
Let $\phi_E = (S^*, V^*, R^*, A_s^*, R_s^*)$ denote the scam-endemic equilibrium (SEE), where all derivatives in model~\eqref{eq:model} vanish. Solving the steady-state equations yields the following 
\begin{equation}
    \left.
    \begin{aligned}
        S^* &= \frac{\sigma R^*}{\beta A_s^*}, \\
        V^* &= \frac{\beta S^* A_s^*}{\gamma + \psi}, \\
        R^* &= \frac{\gamma V^*}{\sigma}, \\
        A_s^* &\ne 0 \Rightarrow 1 + \frac{\beta \psi S^*}{(\delta - \mu - \lambda)(\gamma + \psi)} = 0 \Rightarrow 1 - \mathcal{R}_0 = 0 \Rightarrow \mathcal{R}_0 = 1.
    \end{aligned}
    \qquad \quad \right\}
\end{equation}

Thus, a transcritical bifurcation occurs at $\mathcal{R}_0 = 1$. The SEE exists and is stable whenever $R_0 > 1$.

\begin{theorem}
The scam-endemic equilibrium $\phi_E$ is globally asymptotically stable in $\Gamma$ for $R_0 > 1$.
\end{theorem}

\begin{proof}
Consider the Volterra-type Lyapunov function:
\begin{align}
    M = \sum_{i=1}^5 \left( \frac{C_i}{C_i^*} - \ln\left(\frac{C_i}{C_i^*}\right) - 1 \right),
\end{align}

where $C_i$ denotes each compartment variable and $C_i^*$ its endemic value. Then $M$ is positive definite and vanishes at equilibrium.

Let $U_i = \frac{C_i^*}{C_i}$. For each $C_i$, we obtain
\begin{equation}
    \left.
    \begin{aligned}
        \frac{dM(S)}{dt} &= -\frac{\beta A_s}{U_S}(1 - U_S)^2 < 0, \\
        \frac{dM(V)}{dt} &= -\frac{\gamma + \psi}{U_V}(1 - U_V)^2 < 0, \\
        \frac{dM(R)}{dt} &= -\frac{\sigma}{U_R}(1 - U_R)^2 < 0, \\
        \frac{dM(A_s)}{dt} &= -\frac{\lambda + \mu - \delta}{U_{A_s}}(1 - U_{A_s})^2 < 0, \\
        \frac{dM(R_s)}{dt} &= 0.
    \end{aligned}
    \qquad \quad \right\}
\end{equation}

Thus,
\begin{align}
    \frac{dM}{dt} = \sum_{i=1}^5 \frac{dM(C_i)}{dt} \le 0,
\end{align}
with equality only at equilibrium. Therefore, $\phi_E$ is globally asymptotically stable when $\mathcal{R}_0 > 1$, which means that every trajectory reaches $\phi_E$ whenever $\mathcal{R}_0 > 1$.
\end{proof}

\section{Numerical simulation and results} \label{NSFD_sect}
\subsection{Parameter estimation}
\noindent We fitted our model to the monthly reported fraud cases in each Canadian province between January 2021 and March 2025, obtained from the Canadian Anti-Fraud Centre~\cite{canadafraud2025}. Parameter estimation was performed using a Bayesian Markov Chain Monte Carlo (MCMC) approach implemented with the MCMCSTAT MATLAB package~\cite{laine_mcmcstat}. This toolbox provides adaptive Metropolis-Hastings sampling with covariance adaptation for efficient exploration of the posterior distribution, without relying on external MATLAB toolboxes~\cite{haario2006dram}. 

The estimated parameters are: $\beta, \gamma, \sigma, \delta, \mu, \lambda, \psi$. The system of non-linear ODEs was numerically solved with MATLAB’s ode15s solver at each MCMC iteration to generate model predictions. Uniform prior distributions were assigned within biologically plausible bounds. The likelihood was formulated as a sum-of-squares error function, equivalent to a Gaussian error model, between observed and predicted victim counts.

We ran 10,000 iterations of the Delayed Rejection Adaptive Metropolis (DRAM) algorithm implemented in MCMCSTAT \cite{haario2006dram}, using an adaptive multivariate Gaussian proposal distribution to improve mixing and convergence \cite{haario2001adaptive}. The posterior distributions of the parameters were analyzed to obtain mean estimates and credible intervals. Posterior predictive simulations were generated by sampling parameter sets from the MCMC chains and simulating the model, confirming the model’s ability to capture the observed dynamics.  The results are shown in Table~\ref{tab:posterior} and Figure~\ref{fig:posteriorfit}, respectively.

\begin{table}[h!]
\centering
\caption{Posterior mean estimates and approximate 95\% credible intervals of model parameters}
\label{tab:posterior}
\begin{tabular}{lccc}
\hline
\hline
\textbf{Parameter} & \textbf{Posterior Mean} & \textbf{2.5\% Quantile} & \textbf{97.5\% Quantile} \\
\hline
\hline
$\beta$   & 0.008425 & 0.005   & 0.012 \\
$\sigma$  & 0.023868 & 0.015   & 0.035 \\
$\gamma$  & 0.059366 & 0.035   & 0.090 \\
$\psi$    & 0.000004 & 0.000001 & 0.00001 \\
$\delta$  & 0.026109 & 0.015   & 0.040 \\
$\mu$     & 0.016590 & 0.010   & 0.025 \\
$\lambda$ & 0.016003 & 0.009   & 0.025 \\
\hline
\hline
\end{tabular}
\end{table}

\begin{figure}[h!]
    \centering
    \includegraphics[width=0.5\linewidth]{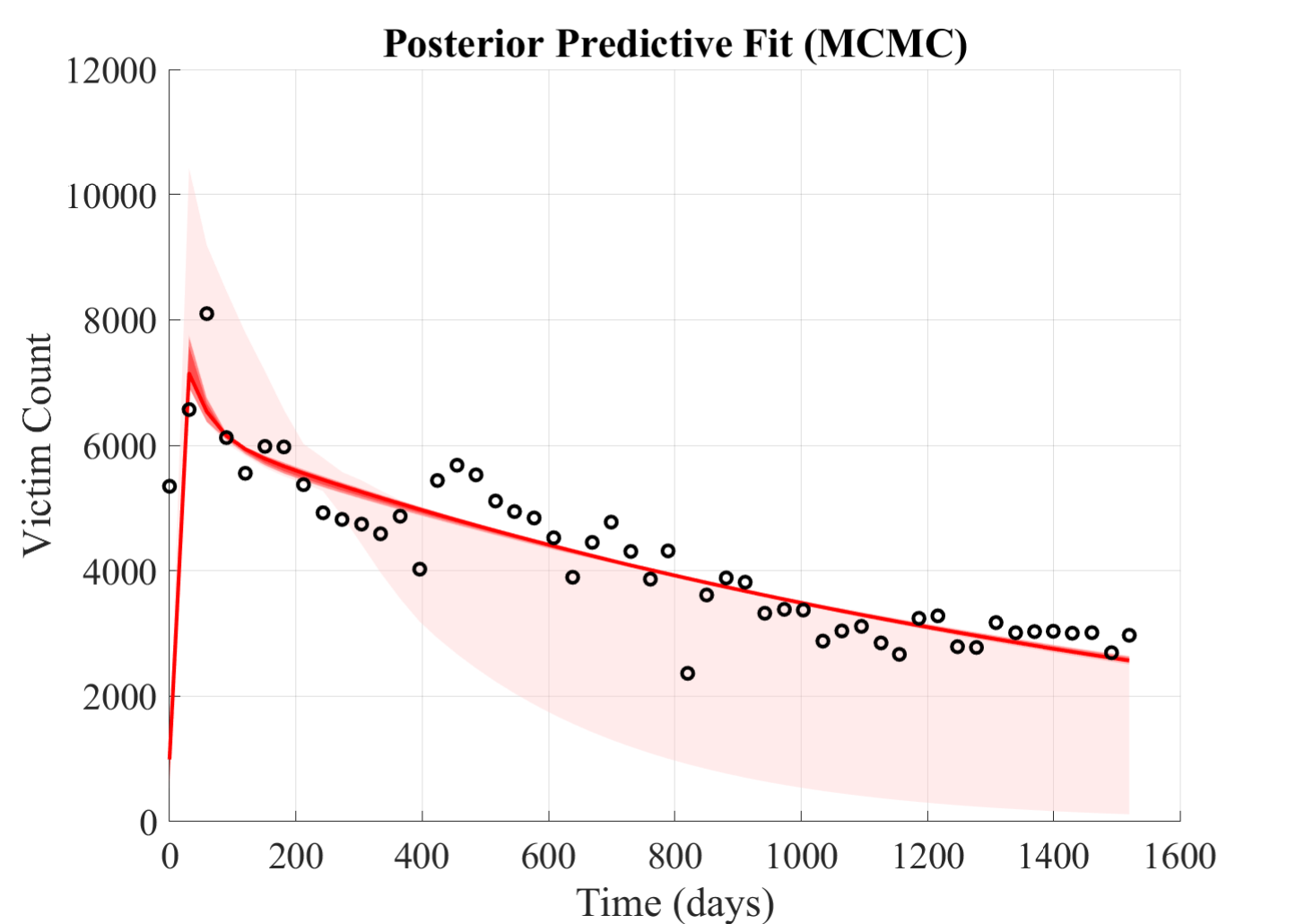}
    \caption{Posterior distributions of the fitted model parameters. The open circles represent the observed data, while the solid lines and shaded areas show the posterior mean and 95\% credible intervals, respectively}
    \label{fig:posteriorfit}
\end{figure}

\subsection{Nonstandard finite difference scheme}
To ensure that the essential qualitative features of the continuous model are retained in its numerical approximation, we employ a structure-preserving discretization technique. Given that the positivity of the continuous dynamical system \eqref{eq:model} has been established in Theorem \ref{thm:positivity}, we adopt a discrete scheme designed to preserve the non-negativity and boundedness of the dynamical system. This is achieved by applying the non-standard finite difference (NSFD) method. Before detailing the specific discretization of \eqref{eq:model}, we briefly outline the core principles of the NSFD approach.

\begin{definition}\label{DefNSFD}
A one-step finite difference scheme for approximating a differential equation is classified as an NSFD scheme if it satisfies at least one of the following conditions in its construction, as described in the seminal work of Mickens \cite{Mickens}:
\begin{enumerate}
\item \label{Rule1} Non-local approximation of terms: Linear or non-linear terms appearing in the differential equation are approximated using combinations of the dependent variable evaluated at more than one grid point (either spatial or temporal), or points distinct from the standard evaluation point for the finite difference approximation. For example, $ u \approx \dfrac{u_{k+1} + u_{k-1}}{2} $, $u^{3} \approx 3u_{k+1} (u_{k})^2 - 2(u_{k})^{3}$.

\item \label{Rule2} Use of non-traditional denominator functions: The standard step size (e.g., $h$ for time) in the denominator of the discrete form of the derivatives is replaced by a function of the step size, denoted here as $\theta(h)$, where $h$ represents the step size. This function $\theta(h)$ must satisfy the property
\begin{equation}
	\lim_{h\to 0} \frac{\theta(h)}{h} = 1 \quad \text{or equivalently} \quad \theta(h) = h + \mathcal{O}(h^{2}) \quad \text{as } h \to 0.
\end{equation}
This ensures that the discrete derivative converges to the continuous derivative as the step size approaches zero. The specific form of $\phi(h)$ is strategically chosen, often dependent on parameters or terms from the differential equation, to preserve the specific dynamical properties of the continuous system. For example, the time derivative $\frac{du}{dt}$ might be approximated by $\displaystyle \frac{u_{k+1} - u_k}{\theta(h)}$ instead of $\frac{u_{k+1} - u_k}{h}$.
\end{enumerate}
\end{definition}

The primary goal of constructing an NSFD scheme is to develop a discrete model whose qualitative dynamics mirrors those of the continuous differential equation, often valid for arbitrary step sizes \cite{Tijani1}.

\subsubsection{The NSFD scheme}
We consider the discrete-time approximation of the continuous model \eqref{eq:model}. Let $\mathbf{\Omega}_k = [S_k, V_k, R_k, A_{s_k}, R_{s_k}]^T$ denote the vector of numerical approximations to the solution $\mathbf{\Omega}(t_k) = [S(t_k), V(t_k), R(t_k), A_s(t_k), R_s(t_k)]^T$ at the discrete time points $t_k = k h$, for $k = 0, 1, 2, \dots$. Here, $h > 0$ represents the constant time step size. The dynamically-consistent discrete scheme for \eqref{eq:model} is given as: 

\begin{equation}\label{eq:nsfd}
\left.
\begin{aligned}
\dfrac{S_{k+1} - S_{k}}{\theta(h)} &= -\beta S_{k+1}A_{s_{k}} + \sigma R_{k}, \\
\dfrac{V_{k+1} - V_{k}}{\theta(h)} &= \beta S_{k}A_{s_{k}} - \gamma V_{k+1} - \psi V_{k+1}, \\
\dfrac{R_{k+1} - R_{k}}{\theta(h)} &= \gamma V_{k} - \sigma R_{k+1}, \\
\dfrac{A_{s_{k+1}} - A_{s_{k}}}{\theta(h)} &= \delta A_{s_{k}} - \mu A_{s_{k+1}} - \lambda A_{s_{k+1}} + \psi V_{k}, \\
\dfrac{R_{s_{k+1}} - R_{s_{k}}}{\theta(h)} &= \lambda A_{s_{k}}, 
\end{aligned}
\qquad \quad \right\}
\end{equation}
where $\theta(h)$ is the non-standard denominator function defined as
\begin{align*}
    \theta(h) = \dfrac{1 - e^{-(\sigma + \mu + \gamma + \lambda + \psi)h}}{\sigma + \mu + \gamma + \lambda + \psi}.
\end{align*}

This specific form for the denominator function, which depends on the combined rate $(\sigma + \mu + \gamma + \lambda + \psi)$ present in some parts of the continuous system, satisfies the consistency property $\theta(h) = h + \mathcal{O}(h^2)$ as $h \to 0$. The use of this non-standard denominator, along with the evaluation of certain terms at the forward time step $k+1$ (e.g., $S_{k+1}$, $V_{k+1}$, $R_{k+1}$, $A_{s_{k+1}}$), are applications of the rules outlined in Definition~\ref{DefNSFD} and are key features of this NSFD scheme designed to ensure its dynamical consistency with \eqref{eq:model} for any step size $h > 0$.

\subsubsection{The NSFD dynamical property }
Here, we established some of the properties of the continuous model \eqref{eq:model} for the discrete scheme \eqref{eq:nsfd}.
\begin{theorem}\label{pos+nsfd}
The NSFD scheme in \eqref{eq:nsfd} is non-negative for any $k \in \mathbb{N}$, whenever $\mathbf{\Omega}(t_0) \geq 0$, each of the parameters $\beta, \delta, \sigma, \mu, \gamma, \lambda, \psi \geq 0$  and 
\begin{equation}
S_{k+1} \geq 0, \quad   V_{k+1} \geq 0, \quad R_{k+1} \geq 0, \quad A_{s_{k+1}} \geq 0, \quad R_{s_{k+1}} \geq 0.
\end{equation}
\end{theorem}

\begin{proof}
We can establish that, in a compact form, the NSFD scheme \eqref{eq:nsfd} becomes

\begin{equation}\label{eq:nsfd1}
\left.
\begin{aligned}
S_{k+1} &= \dfrac{S_{k} + \theta(h)\sigma R_{k}}{1 + \theta(h) \beta A_{s_{k}}}, \\
V_{k+1} &= \dfrac{V_{k} + \theta(h) \beta S_{k} A_{s_{k}}}{1 + \theta(h)\gamma + \theta(h)\psi}, \\
R_{k+1} &=  \frac{R_k + \theta(h)\gamma V_k}{1 + \theta(h)\sigma},\\
A_{s_{k+1}} &=  \frac{A_{s_k}(1 + \theta(h)\delta) + \theta(h)\psi V_k}{1 + \theta(h)(\mu + \lambda)}, \\
R_{s_{k+1}} &= R_{s_{k}} + \theta(h)\lambda A_{s_{k+1}}.
\end{aligned}
\qquad \qquad \right\}
\end{equation}
Thus, starting with non-negative initial conditions $\mathbf{\Omega}_0 \ge \mathbf{0}$, by induction, the scheme \eqref{eq:nsfd} guarantees that $\mathbf{\Omega}_k \ge \mathbf{0}$ for all $k \ge 0$.
\end{proof}

\begin{prop}\label{gas+nsfd}
The following statements establish the existence and stability of the scam-free equilibrium for the discrete model.
\begin{enumerate}
\item[(i)] The discrete scheme satisfies the scam-free equilibrium (SFE) by setting
\begin{align*}
    S_{k+1}=S_k=N, \; V_{k+1}=V_k=0, \; R_{k+1}=R_k=0, \; A_{s_{k+1}}=A_{s_k}=0, \; R_{s_{k+1}}=R_{s_k}=0.  
\end{align*}

\item[(ii)] The NSFD scheme in \eqref{eq:nsfd} is locally asymptotically stable with respect to the scam-free equilibrium if and only if $R_{0} < 1$ and $\mu + \lambda > \delta$.
\end{enumerate}
\end{prop}

\subsubsection{Simulation and discussion of the NSFD results}\label{Discus_sect}
This section explains the parametric analysis of the model parameters in the continuous differential equations \eqref{eq:model}. The absence of variation in the active scammer population for different values of $\gamma$ is expected, as seen in Figure~\ref{fig_gamma}. This is because $\gamma$, which represents the rate at which victims become resistant to the scammer's tactics, does not physically affect the active scammer population. This is further supported by the fact that $\gamma$ does not appear in the active scammer differential equation. Its indirect influence is also negligible given that $\psi \approx 0$. In stark contrast to the scammers population, the victim population is directly and significantly affected by the recovery rate $\gamma$, and the graph illustrates this phenomenon. Therefore, $\gamma$ can be understood as the efficiency of the recovery process.
Figure~\ref{fig_gamma} clearly shows that as $\gamma$ increases, the victim population decreases over time. A higher $\gamma$ creates a more powerful outflow from the victim pool.
\begin{figure}[htpb!]
	\centering
    \includegraphics[width=0.45\textwidth]{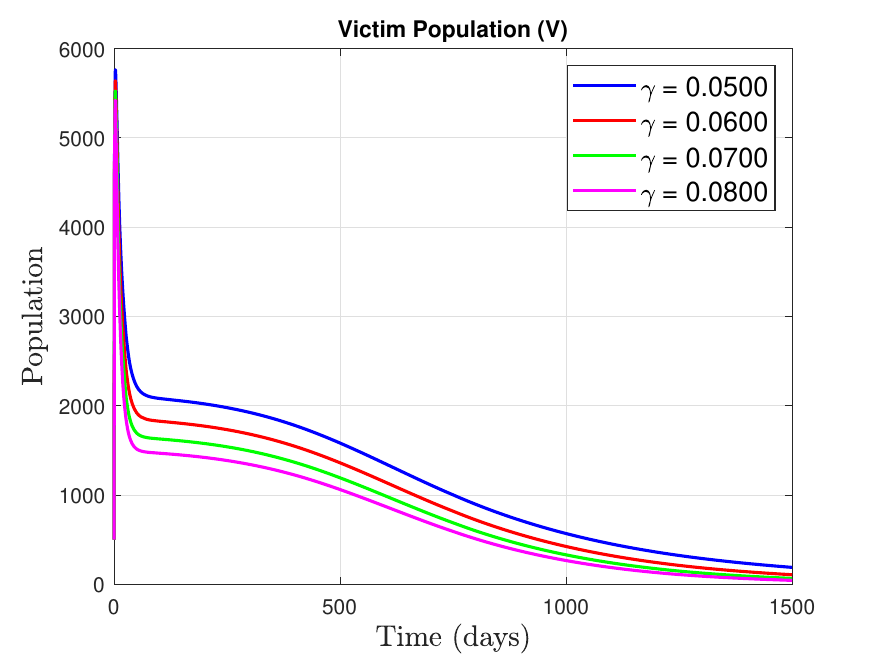}
    \includegraphics[width=0.45\textwidth]{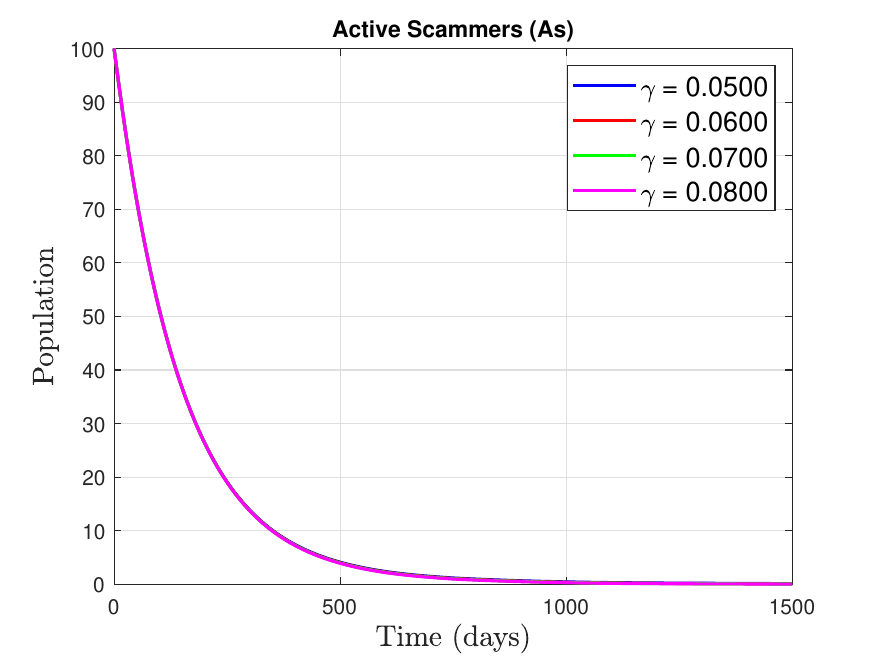}
	\caption{Parametric significance of the rate at which victims become resistant to scams for the populations of victims and active scammers}
    \label{fig_gamma}	
\end{figure}
Figure~\ref{fig_lambda} reveals the consequential impact of law enforcement or natural removal of scammers on the victim population. While the parameter $\lambda$ does not appear directly in the deterministic equation for the victims population, its influence is transmitted through the coupled nature of the system. The graph shows that a higher enforcement rate $\lambda$ leads to a more rapid decline in the victim population. This suggests that as active scammers are removed or arrested, someone in the victim population has a sigh of relief or is free from scams, perhaps at that point in time. Figure~\ref{fig_lambda} also illustrates the direct and potent effect of the law enforcement rate, $\lambda$, on the active scammer population. The graph shows a clear, monotonic relationship: as the rate of apprehension $(\lambda)$ increases, the time required for the scammer population to collapse to near-zero levels decreases dramatically. For instance, with a relatively low enforcement rate of $\lambda = 0.04$, the scammer population persists for over $600$ days before becoming negligible. In contrast, when the enforcement rate is increased to $\lambda = 0.07$, the population is effectively eradicated in fewer days. Figure~\ref{fig_mu} illustrates the effect of voluntary scammer removal and the longevity of the active scammer population, governed by the parameter $\mu$. Both $\mu$ and $\lambda$ reduce the active scammer population, making their effects qualitatively similar. However, $\mu$ demonstrates a significantly greater impact on the scam-victim dynamics for the population. For instance, a $\mu$ value of $0.003$ eliminates the scammer population in approximately $600$ days, an outcome that requires a much larger $\lambda$ value of approximately $0.04$, as shown in Figure~\ref{fig_lambda}. 
\begin{figure}[htpb!]
	\centering
    \includegraphics[width=0.45\textwidth]{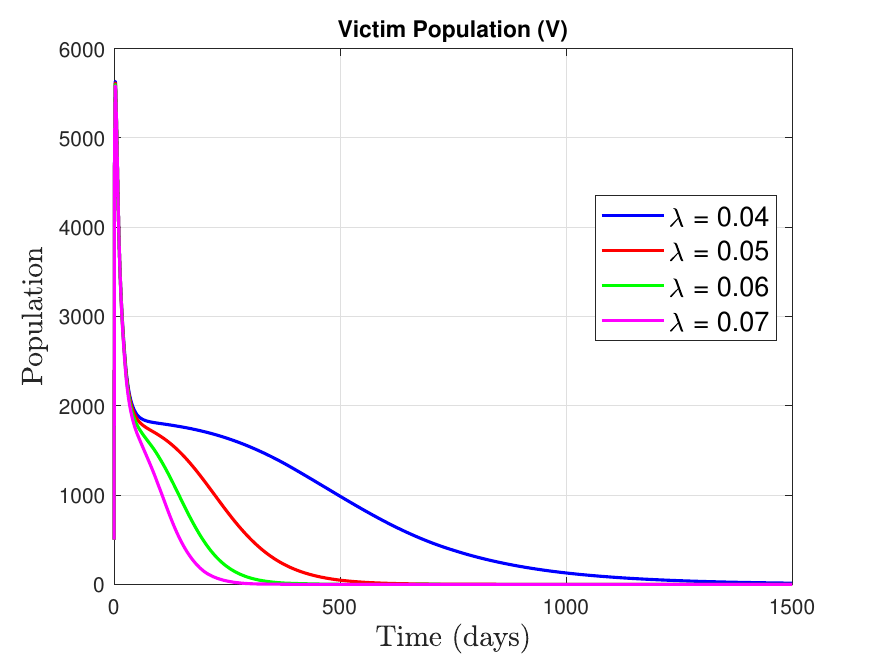}
    \includegraphics[width=0.45\textwidth]{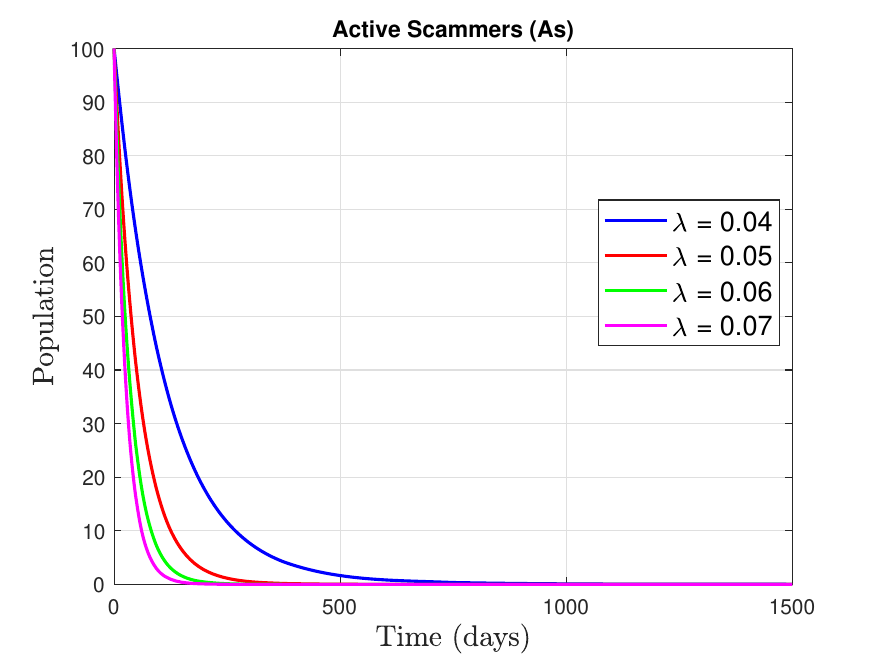}
	\caption{Parametric significance of the rate at which scammers are caught for the populations of victims and active scammers}
    \label{fig_lambda}	
\end{figure}
\begin{figure}[htpb!]
	\centering
    \includegraphics[width=0.45\textwidth]{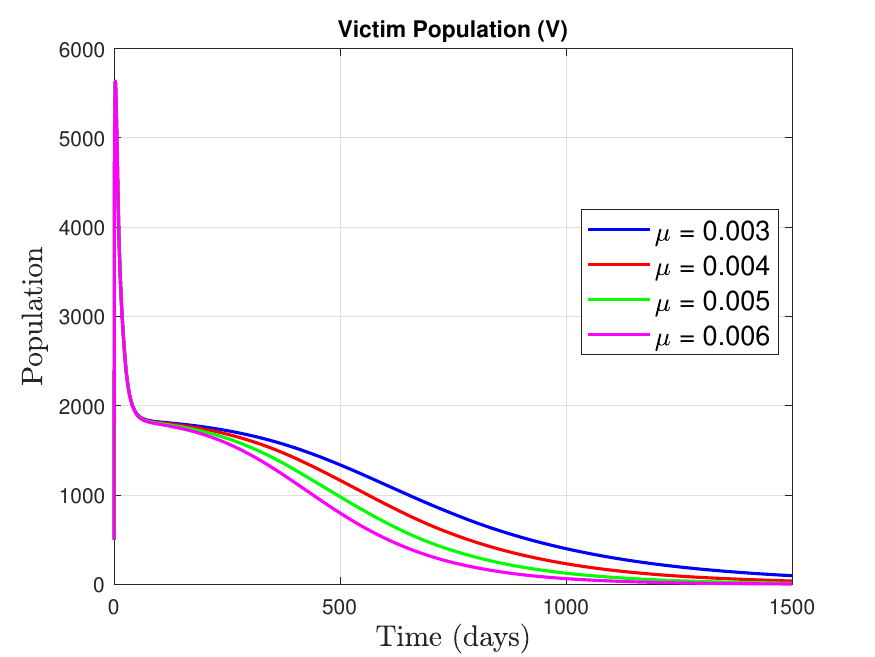}
    \includegraphics[width=0.45\textwidth]{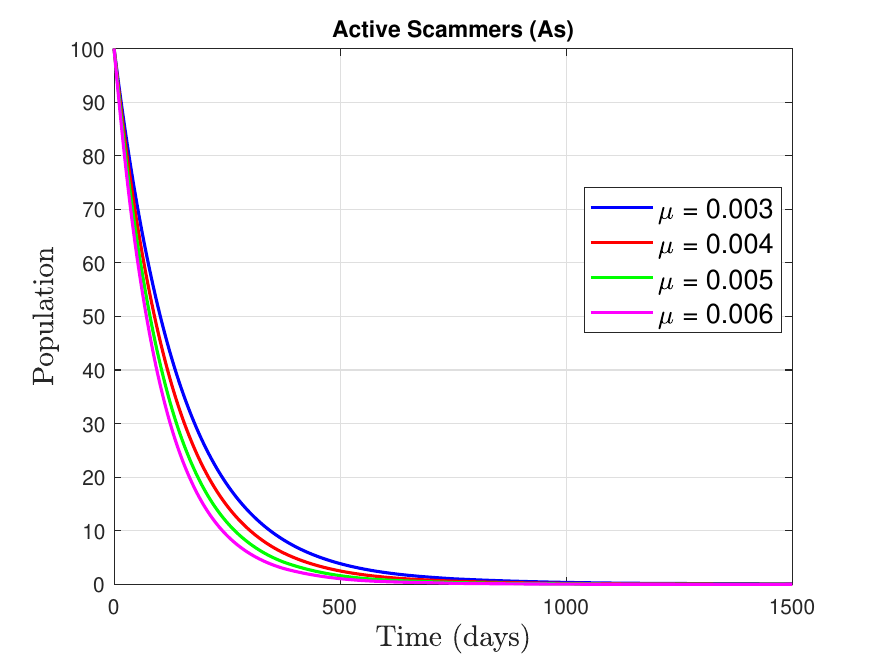}
	\caption{Parametric significance of the scammer attrition rate for the victim and active scammer populations}
    \label{fig_mu}	
\end{figure}
From a physical standpoint, this demonstrates that the internal stability of the scamming enterprise is highly sensitive to the turnover rate of its population. Factors that increase $\mu$, such as the psychological stress of the activity, diminishing returns, social programs, economic incentives for legitimate work, public awareness, or the availability of viable legitimate alternatives, can be as effective at dismantling the criminal network as external pressures like law enforcement or death. The model suggests that making scamming an unappealing or unsustainable career is a powerful internal lever for its eradication. Figure~\ref{fig_delta} provides a parametric analysis of the scammer recruitment rate, $\delta$. This parameter represents the ability of the scammers' population to grow, for instance, through the recruitment of new members from the overall population, independent of the victim population. The simulations explore how variations in this recruitment efficacy influence the temporal dynamics of both the victim and active scammer populations. The results in Figure~\ref{fig_delta} show a pronounced relationship: a higher recruitment rate directly correlates with increased persistence of the scammer population. While all previously simulated scenarios culminate in the eventual collapse of the scammer population, a higher $\delta$ value significantly delays this outcome. For example, a low recruitment rate of $\delta = 0.010$  results in a rapid decay, whereas a higher rate of $\delta = 0.035$ allows the scammer population to remain active for a substantially longer period.
\begin{figure}[htpb!]
	\centering
    \includegraphics[width=0.45\textwidth]{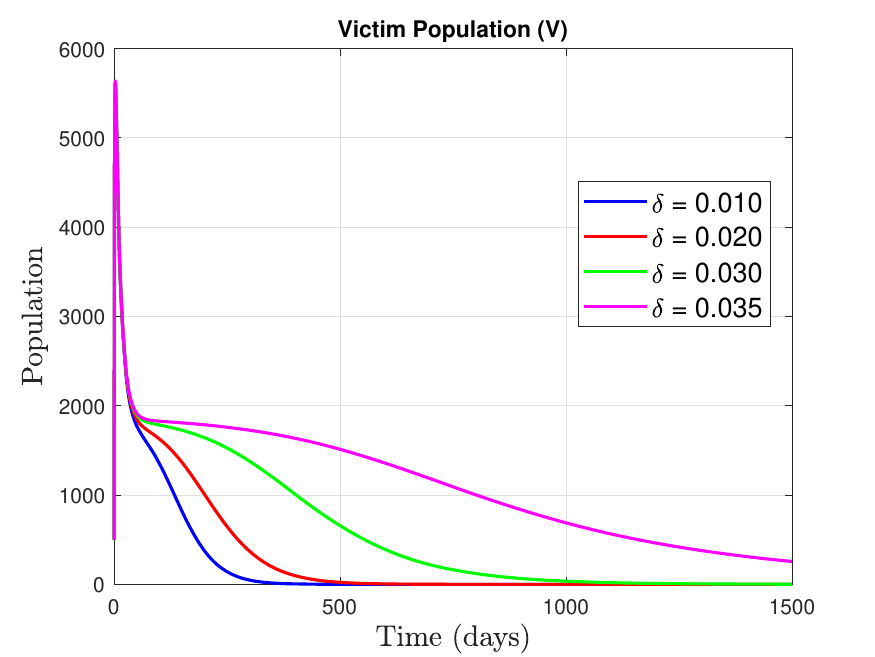}
    \includegraphics[width=0.45\textwidth]{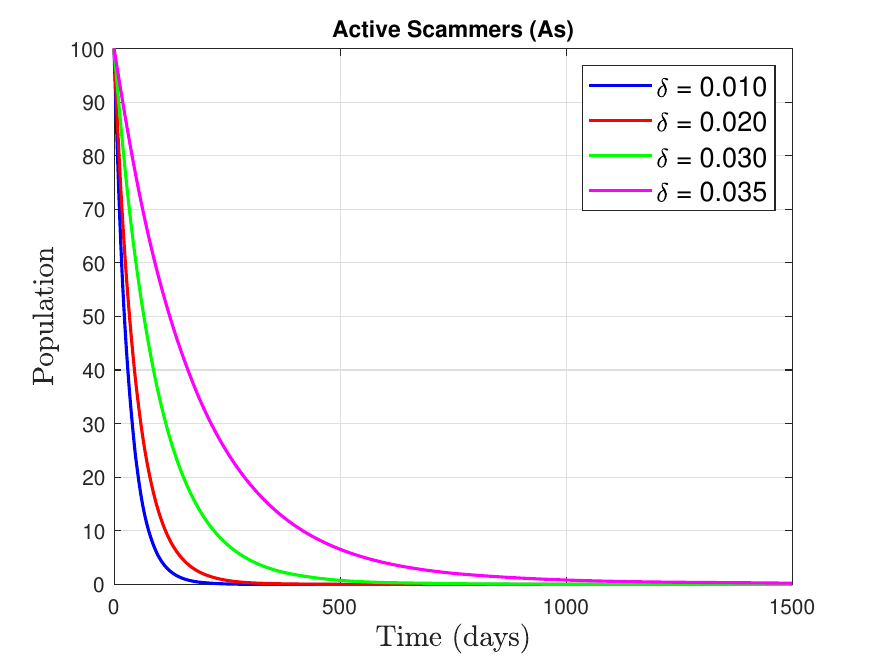}
	\caption{Parametric significance of the scammer recruitment rate for the victim and active scammer populations.}
    \label{fig_delta}	
\end{figure}
\begin{figure}[htpb!]
	\centering
    \includegraphics[width=0.45\textwidth]{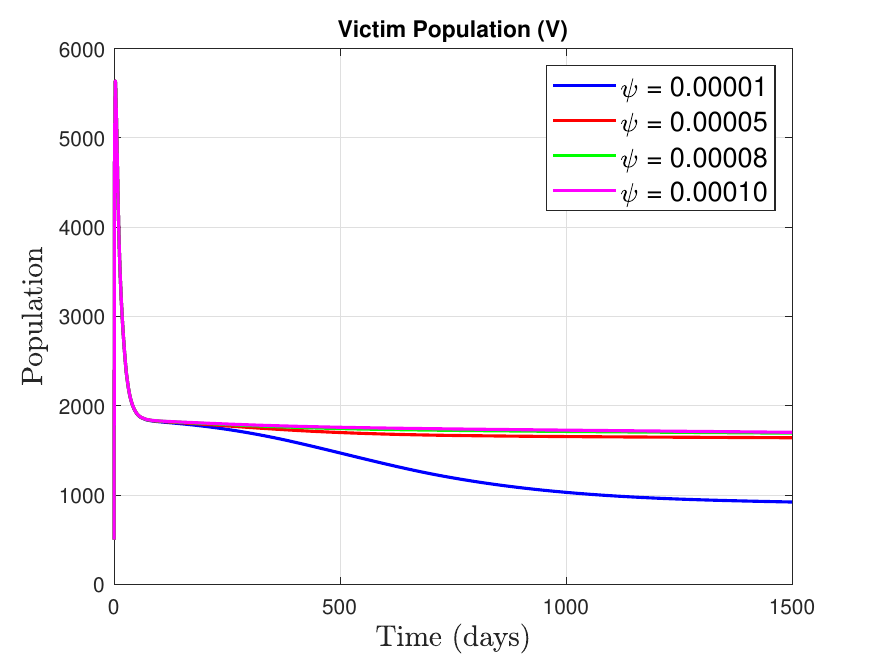}
    \includegraphics[width=0.45\textwidth]{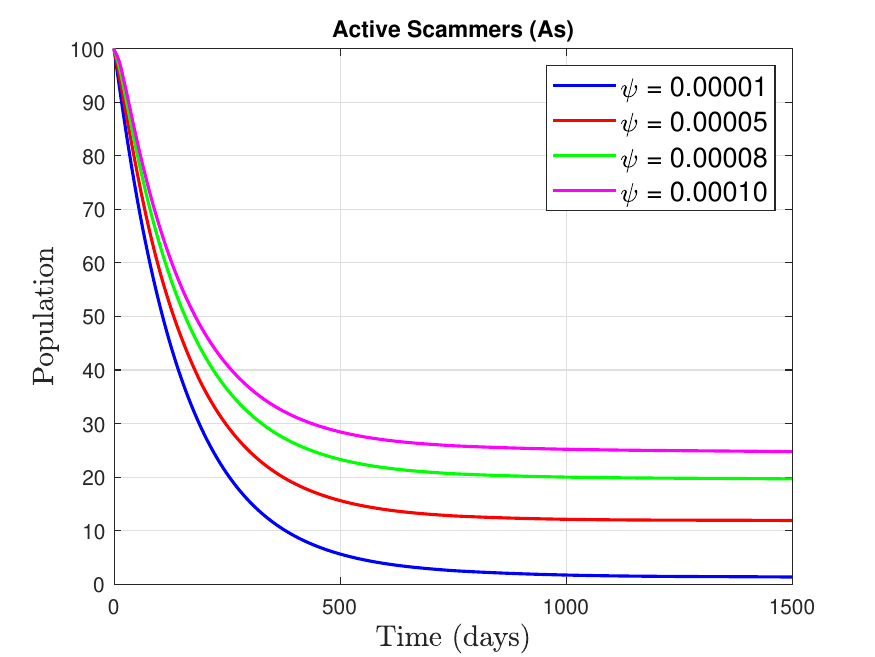}
	\caption{Parametric significance of the rate at which victims transition to scammers for the populations of victims and active scammers}
    \label{fig_psi}	
\end{figure}
A higher $\delta$ value not only prolongs the existence of the active scammer population but also, as a direct result, extends the period during which individuals are victimized. Consequently, the total time over which the society experiences a significant victim population is extended. This study suggests that intervention efforts should focus on increasing scammer removal rates and disrupting their recruitment pathways (i.e., decreasing $\delta$). Policies and actions aimed at reducing $\delta$ value, such as monitoring and shutting down recruitment channels, or addressing the socioeconomic factors that make scamming an attractive option, could critically undermine the resilience of the scammer population. By reducing the scammers' ability to recruit new members, the active scammer population decline can be accelerated, thereby shortening the duration of harm experienced by the victim population. Unlike earlier simulations where the scammer population invariably collapsed to zero, varying $\psi$ demonstrates the potential for the scammer population to achieve a non-zero, stable equilibrium as seen in Figure~\ref{fig_psi}. The graph in Figure~\ref{fig_psi} shows a clear and strong positive correlation where a higher victim-to-scammer transition $\psi$ results in a more sustained population of active scammers. The results also depict that a higher victim-to-scammer transition rate $\psi$ leads to a lower equilibrium level of active victim population. A superficial observation of a decreasing victim population could be misinterpreted as a positive trend, when in fact it might be symptomatic of a more dangerous systemic shift towards a self-perpetuating scamming system. Measures aimed at providing victims support, counseling, and alternative pathways are crucial to keep $\psi$ low, thereby ensuring that the scamming operation remains fundamentally unsustainable and preventing its becoming an endemic scenario.
\begin{figure}[htpb!]
	\centering
    \includegraphics[width=0.45\textwidth]{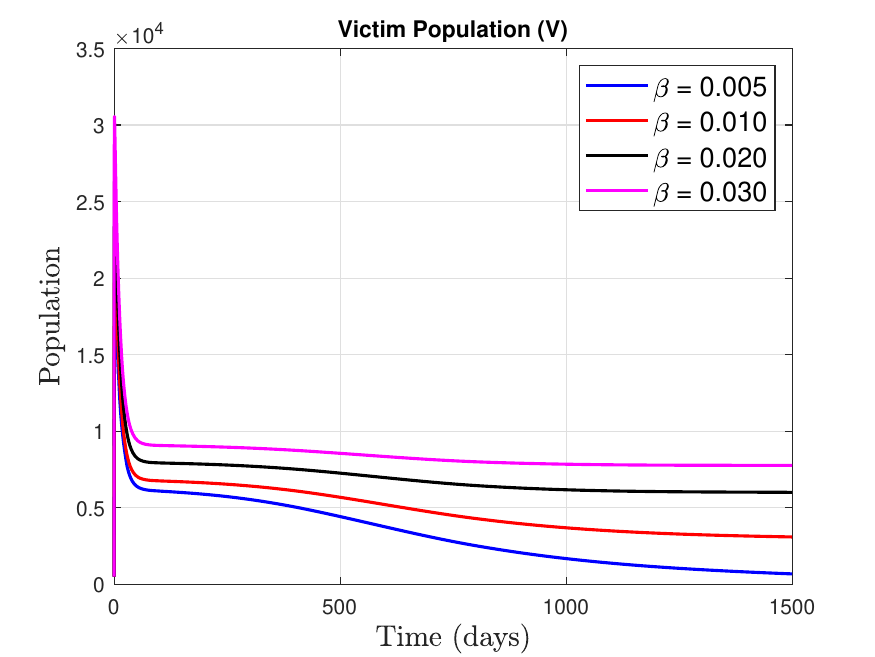}
    \includegraphics[width=0.45\textwidth]{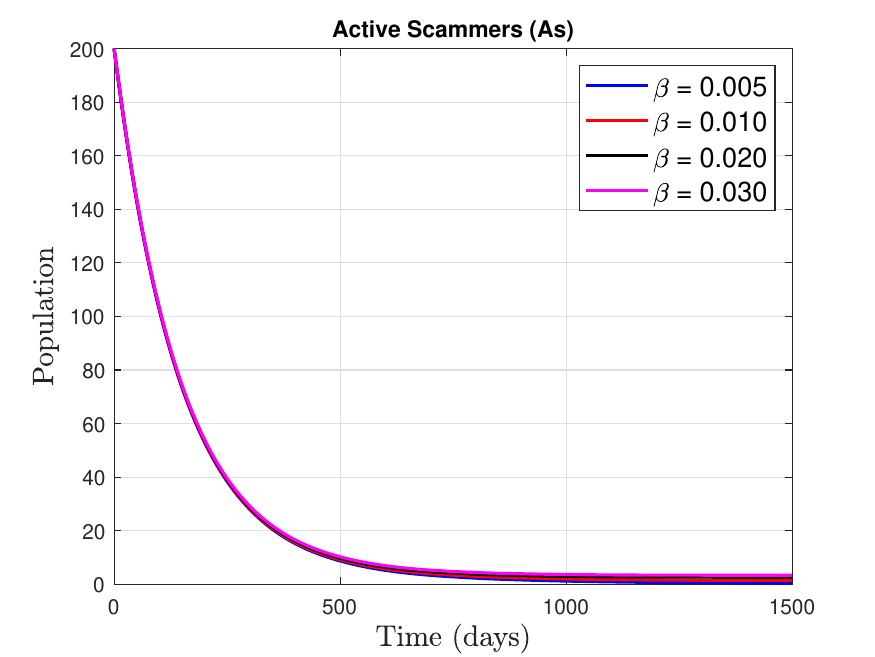}
	\caption{Parametric significance of the rate at which susceptibles fall victim for the populations of victims and active scammers.}
    \label{fig_beta}	
\end{figure}
The parameter $\beta$, which quantifies the rate at which susceptible individuals move into the victim compartment, is discussed in Figure~\ref{fig_beta}. $\beta$ is a fundamental parameter, as it represents the primary mechanism of scamming within the scam-victim dynamics. Figure~\ref{fig_beta} depicts a direct and intuitive relationship between the victimization rate and the size of the victim population. The figure shows that a higher value of $\beta$ results in a significantly higher steady-state (equilibrium) population of victims. Following an initial sharp drop as the system adjusts from its initial conditions, each simulation converges to a stable victim population over time. The susceptible-victim rate $\beta$ has a negligible effect on the long-term fate of the active scammer population. The simulations show that regardless of how effective the scams are, the scammer population undergoes a rapid collapse from its initial state of 200 individuals, approaching zero over the same timescale in all scenarios.

\subsection{Sensitivity analysis of the model}\label{SA_sect}
In the mathematical modelling of complex dynamical systems, sensitivity analysis plays a pivotal role in understanding how variations in model parameters influence the behavior of the system outputs. Since most models are constructed based on parameters estimated from empirical data, drawn from literature, or assumed based on theoretical considerations, there is an inherent uncertainty associated with these parameter values. Sensitivity analysis provides us with a systematic approach for quantifying the impact of such uncertainties on model predictions, or in this case, quantifying the effects of the model parameters on the model output. 

Sensitivity analysis can generally be classified into two broad categories, namely, global and local sensitivity analyses. Global sensitivity analysis explores the influence of parameters over a wide range of values, often requiring extensive computational effort, and is suited for identifying cross-interaction and non-linear effects across the entire parameter space. In contrast, local sensitivity analysis focuses on quantifying the behavior of the model output in the immediate neighborhood of a nominal set of parameter values. This is done by computing the partial derivatives of the model output with respect to each parameter, thus quantifying the rate of change of the output due to an infinitesimally small perturbation in the individual parameters around the nominal values of the parameter, while keeping all other parameters fixed. The main advantage of the local sensitivity analysis lies in its simplicity and interoperability, as it gives an understanding of which parameters the model is most sensitive to at a given point in the parameter space. 

\subsubsection{Local sensitivity analysis of \texorpdfstring{$R_0$}{R0}}
In this section, we look at the sensitivity of the basic reproduction number in \eqref{eq:R0}, which represents the potential for the spread of scamming activities within the population, 
\begin{align}
    \mathcal{R}_{0} = \frac{\beta N \psi}{(\gamma + \psi)(\mu + \lambda - \delta)},
\end{align}
where $\beta, \psi, \gamma, \mu, \lambda$ and $\delta$ are model parameters defined in Table~\ref{tab:paramdescr}. To perform the local sensitivity analysis, we compute the partial derivative of $R_0$ with respect to each of the parameters, and evaluate at the nominal values, i.e., the fitted values in Table~\ref{tab:paramdescr}. These derivatives measure the instantaneous rate of change of $R_0$ in response to small changes in each parameter. However, for ease of reference and interoperability, we compute the normalized sensitivity indices of $R_0$ with respect to the parameters. The normalized sensitivity index of $R_0$ with respect to a parameter, say $p$, is defined as 
\begin{align}
    S_p = \frac{\partial \mathcal{R}_0}{\partial p} \cdot \frac{p}{\mathcal{R}_0}.
\end{align}
This dimensionless quantity represents the percentage change in $\mathcal{R}_0$ per percentage change in $p$, which gives a proportional effect of each parameter on $\mathcal{R}_0$. A positive sensitivity index means that increasing the parameter leads to an increase in $\mathcal{R}_0$, while a negative index implies that an increase in the parameter results in a decrease in $\mathcal{R}_0$. For the $\mathcal{R}_0$ under consideration, the normalized local sensitivity indices with respect to the parameters $\beta, \psi, \gamma, \mu, \lambda$ and $\delta$ are, respectively, defined as follows:
\begin{equation}
        S_\beta = 1, \quad S_\psi = \frac{\gamma}{\gamma + \psi}, \quad S_\gamma = -\frac{\gamma}{\gamma + \psi}, \quad S_\mu = - \frac{\mu}{\mu + \lambda - \delta}, \quad S_\lambda = -\frac{\lambda}{\mu + \lambda - \delta}, \quad S_\delta = \frac{\delta}{\mu + \lambda - \delta}.
\end{equation}

                                                                                                                                                                                                                                                                                                                                                                                            We evaluate these indices at the normalized local sensitivity index with respect to each parameter at the nominal values (fitted values in Table~\ref{tab:paramdescr}), and present the results in Table~\ref{tab:sensitivityindex}. The sensitivity of $R_0$ with respect to the rate at which the susceptible population falls victim ($\beta$) is exactly one, indicating that the potential for the proliferation of scamming activities is linearly proportional to the susceptibility rate of the population. A $1\%$ increase in $\beta$ leads to a $1\%$ increase in $R_0$, confirming that the susceptibility rate is a direct driver of scam propagation in the population. Interestingly, the similarity index of $R_0$ with respect to $\psi$, the rate at which scam victims joined active scammers, is nearly equal to $1$, despite the fact that the nominal value of $\psi$ is extremely small. This result suggests that even minimal rates of victim defection into scamming activities can have a disproportionately large impact on the potential for scam activities or scam propagation. The model structure magnifies the influence of $\psi$ because it appears in both the numerator and denominator of $R_0$, albeit in slightly different roles. Most notably, the parameters, $\mu, \lambda$, and $\delta$, all associated with scammer dynamics, have the largest magnitude sensitivity indices. In particular, $\delta$, the scammer recruitment rate, has the most dominant influence on $R_0$, with a sensitivity index of $4.0267$. This implies that a modest increase in recruitment quadruples the potential for scamming activities, highlighting recruitment as the primary driver of scam proliferation. Conversely, increases in scammer dropout rate, $\mu$, and the arrest rate $\lambda$ have substantial negative effects on the proliferation of scamming activities, with sensitivity indices of $-2.5586$ and $-2.4681$, respectively. These findings suggest that interventions aimed at disincentivizing continued scam involvement, such as through social rehabilitation and job alternatives, or enhancing law enforcement through arrest and prosecution, are among the most effective strategies for reducing scam activities. The sensitivity index of $R_0$ with respect to $\gamma$ is very close to $-1$, which means that a $1\%$ increase in $\gamma$ leads to a $\approx 1\%$ decrease in $R_0$. Mathematically, the near-linearity (index near $\approx 1$) reflects the fact that $\gamma \gg \psi$, so $-\sfrac{\gamma}{(\gamma+\mu)} \approx -\sfrac{\gamma}{\gamma} = -1$. This makes $\gamma$ one of the most influential control levers in the model, as it dictates how many victims escape the scam cycle permanently, thereby either increasing the population that is resistant to scamming activities or reducing the population that can be converted to scammers, inevitably, reducing scamming activities.

\begin{table}[h!]
\centering
\captionsetup{justification=centering,,margin=1cm}
\caption{Normalized local sensitivity indices of $R_0$ with respect to $\beta, \psi, \gamma, \mu, \lambda$ and $\delta$ evaluated at the nominal (fitted) values}
 \begin{tabular}{llll}
	\hline \hline
	Parameter & Description  & Nominal value & Sensitivity index \\ \hline \hline
	$\beta$ & Susceptibility rate & $0.008425$ & $+1.0000$ \\
    $\psi$ & Victim defection rate & $0.000004$ & $+0.9999$ \\
    $\gamma$ & Insusceptibility rate & $0.059366$ & $-0.9999$ \\
    $\mu$ & Scammer dropout rate & $0.016590$ & $-2.5586$ \\
    $\lambda$ & Scammer arrest rate & $0.016003$ & $-2.4681$ \\
    $\delta$ & Scammer recruitment rate & $0.026109$ & $+4.0267$ \\
	\hline\hline
\end{tabular}
\label{tab:sensitivityindex}
\end{table}

These results underline the critical role of scammer behavior and transitions in shaping the dynamics of the deterministic model. While traditional preventive strategies may focus on reducing victimization rate, $\beta$, this analysis shows that the most effective leverage point lies in scammer influx and retention. Reducing the scammer recruitment rate, increasing the rate at which scammers quit, or improving the rate at which scammers are arrested and prosecuted are all more potent at controlling the potential for scamming activities than interventions on the victim side. Nevertheless, mechanisms, such as awareness, digital literacy, and scam education programs, that help victims learn and avoid scams in the future can almost proportionally reduce the chance of scamming activities happening. So, investing in victim recovery and resilience is as crucial as enforcement or prevention efforts. Furthermore, the high sensitivity of $R_0$ to $\psi$, even at small values, indicates that preventing the defection of victims to active scammer roles should also be equally the focus of public awareness and support programs. 

\subsubsection{Global analysis of victims and active scammers populations}
Among the global sensitivity analysis techniques, the partial rank correlation coefficient (PRCC) is widely used to quantify the monotonic relationship between model inputs and outputs while controlling for the effects of other parameters. Unlike the local sensitivity indices, PRCC is non-parametric and robust against model nonlinearities. It measures the correlation between the ranks of each input parameter and the model output of interest, partialling out the influence of other covariates through regression. The resulting coefficients indicate the direction and strength of the correlation, while the accompanying $p$-values assess the statistical significance of the correlation coefficients. In this analysis, PRCC was applied to two key variables, the accumulation of active scammers and victims, over a period, $\Omega$, defined, respectively, as
\begin{align}
    A_s^\dagger = \int_\Omega A_s(t) dt \quad \text{and} \quad V^\dagger = \int_\Omega V(t) dt.
\end{align}

The analysis was conducted using $2500$ latin hypercube samples spanning plausibly relevant ranges presented in Table~\ref{tab:prcc}, along with the correlation coefficients of $A_s^\dagger$ and $V^\dagger$ with respect to $\beta, \gamma, \sigma, \delta, \mu, \lambda$ and $\psi$. For this analysis, we take $\Omega$ to span $1520$ days, which is approximately $4$ years and $52$ days. 

\begin{table}[h!]
\centering
\captionsetup{justification=centering, margin=1cm}
\caption{Partial rank correlation coefficients of $A_s^\dagger$ and $V^\dagger$ with respect to $\beta, \gamma, \sigma, \delta, \mu, \lambda$, and $\psi$.}
\label{tab:prcc}
\begin{tabular}{llcccc}
\hline\hline
 &  & \multicolumn{2}{c}{$A_s^\dagger$} & \multicolumn{2}{c}{$V^\dagger$} \\ \cline{3-4} \cline{5-6} 
   Parameter       &   Range    & PRCC & $p$-value & PRCC & $p$-value \\
\hline\hline
 $\beta$   & $[0, 0.1]$  & $+0.0917$ &  $4.35 \times 10^{-6}$ & $+0.1133$ &  $1.33 \times 10^{-8}$ \\
 $\gamma$ & $[0,0.1]$ & $-0.0772$ & $1.11 \times 10^{-4}$ & $-0.0942$ & $2.39 \times 10^{-6}$ \\ 
$\sigma$ & $[0,0.2]$ & $+0.1159$ & $6.25 \times 10^{-9}$ & $0.1415$ & $1.19 \times 10^{-12}$ \\
$\delta$ & $[0, 0.1]$ & $+0.8086$ & $0.0000$ & $+0.0178$ & $0.3732$ \\
$\mu$ & $[0,0.1]$ & $-0.8020$ &  $0.0000$ &  $+0.0170$ & $0.3958$ \\
$\lambda$ & $[0,1.0]$ & $-0.9820$ & $0.0000$ &  $-0.0481$ & $1.62 \times 10^{-2}$ \\
$\psi$ & $[0, 0.1]$ & $+0.2159$ & $9.53 \times 10^{-28}$ & $-0.9859$ & $0.0000$ \\
\hline \hline
\end{tabular}
\end{table}

\begin{figure}[htpb!]
	\centering
    \includegraphics[width=\textwidth]{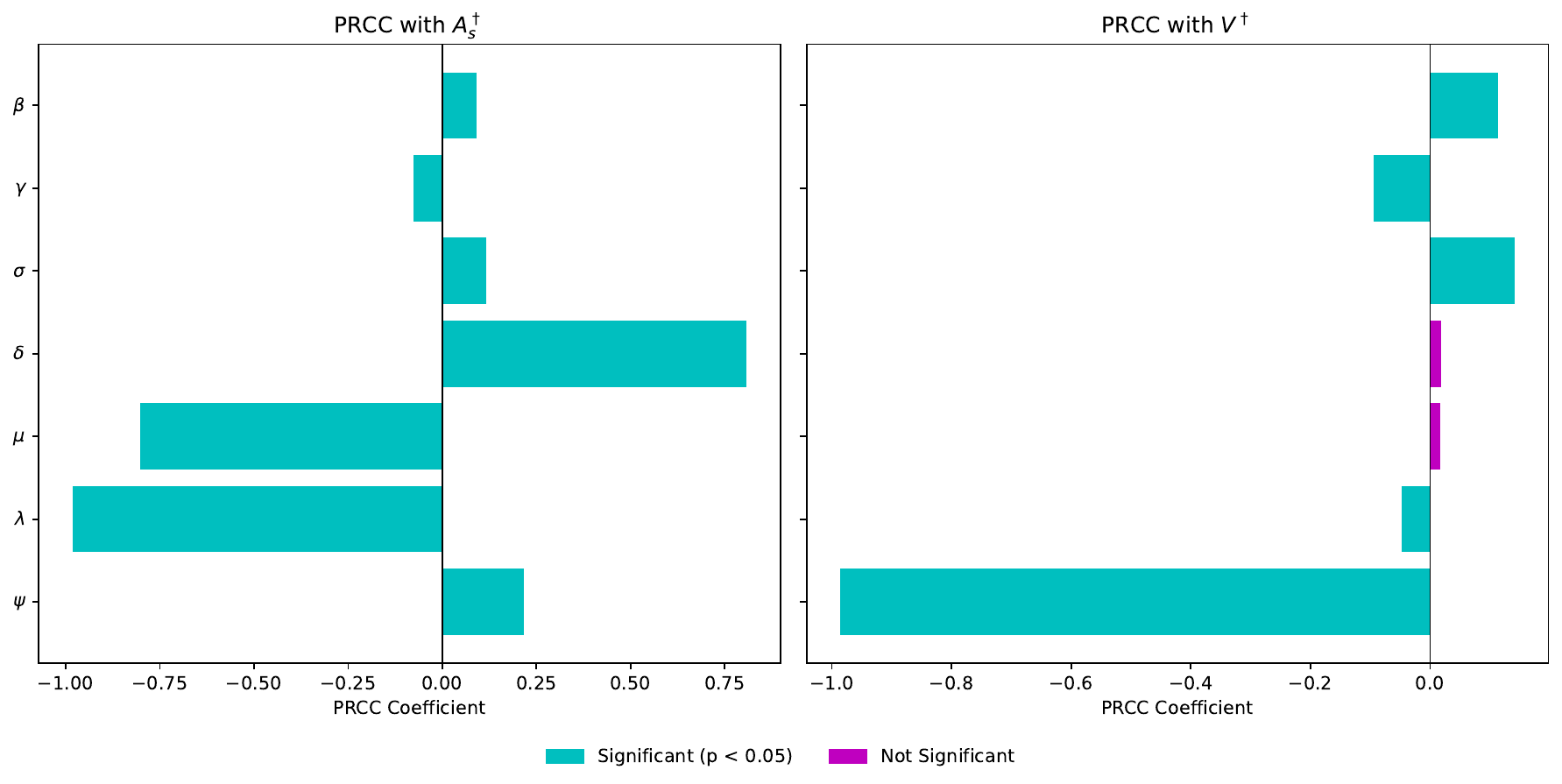}
	\caption{Partial rank correlation coefficients (PRCC) of model parameters with respect to accumulated active scammers ($A_s^\dagger$) and accumulated victim count ($V^\dagger$). Parameters with statistically significant influence ($p < 0.05$) are shown in cyan, while non-significant parameters are shown in magenta.}
    \label{fig_prcc_plot}	
\end{figure}

From the results presented in Table~\ref{tab:prcc} and Figure~\ref{fig_prcc_plot}, the most influential parameters on accumulated active scammers are:
\begin{itemize}
    \item $\delta$ (recruitment rate): showing a strong positive correlation (PRCC $=+0.8086$), indicating that higher recruitment directly increases scammer proliferation. 
    \item $\mu$ (dropout rate): showing very strong negative correlation $(-0.8020)$, implying that improving scammer exit substantially reduces the scammers population, and by extension, scamming activities. 
    \item $\lambda$ (arrest rate): shows the most significant negative correlation $(-0.9820)$, suggesting that effective law enforcement almost linearly suppresses the total number of active scammers.
\end{itemize}

The rate at which some former victims joined active scammers, $\psi$, has a moderate positive influence $(+0.2159)$, indicating that increased defection of victims into scammer roles contributes meaningfully to the increase in the number of active scammers and invariably to the growth of scamming activities. Although statistically significant, parameters such as $\beta, \gamma, \sigma$ have weak correlation, consistent with more indirect roles in the long-term dynamics of the number of active scammers and/or scamming activities. However, the most influential parameter on the accumulated victims during the period considered is $\psi$, which has an extremely negative correlation $(-0.9859)$, indicating that as more victims defect into scamming, the victim compartment depletes drastically, likely due to a shift of population from victim to scammer class. On the one hand, parameters $\beta$ and $\sigma$ show a moderate positive influence on total victims. As expected, increased susceptibility and loss of resistance increase the number of victims. $\gamma$, on the other hand, has a moderate negative impact, consistent with the protective effect of gaining resistance against scammers. Finally, $\delta, \mu, \lambda$ have small or insignificant correlations with the accumulation of victims. This suggests that scammer turnover dynamics has little direct effect on victim pool sizes beyond their influence on scam propagation rates. 

The PRCC results align closely with the normalized sensitivity analysis of $R_0$, but provide a deeper insight into the time-accumulated impact of each parameter on the active scammer and victim population groups. The key takeaways from these analyses include, but are not limited to, the following:
\begin{itemize}
    \item enforcement and scammer attrition are critical for suppressing the prevalence of active scammers, and
    \item psychosocial and behavioral interventions that prevent victim defection (very low $\psi$) are very important for protecting vulnerable populations. 
\end{itemize}

\section{Conclusion}\label{Conc_sect}
The rise of cybercrime has led to an alarming increase in scams, with cyber-enabled fraud becoming increasingly prevalent and sophisticated, leaving many vulnerable to financial and personal losses. This study proposed a continuous deterministic model of scam-victim dynamics. Following an epidemiological approach, the dynamical model divides the population into five compartments: susceptible, victim, recovered, active scammer, and recovered scammer. The theoretical properties of the model are analyzed, revealing that the system is positive-definite, bounded, and exhibits a transcritical bifurcation when the basic reproduction number, $R_0$, is equal to one. We propose a structure-preserving non-standard finite difference (NSFD) scheme to ensure the positivity of the model solutions. Furthermore, sensitivity analysis (PRCC) is conducted to assess the model's robustness and determine the influence of each parameter on the system dynamics. The findings of this study include, but are not limited to, the following:
\begin{itemize}
    \item Effective enforcement measures and a high rate of scammer attrition are essential to mitigate the prevalence of active scammers.
    \item The phenomenon of victims becoming scammers, even at low rates, contributes disproportionately to the expansion and persistence of fraudulent activities.
    \item Increasing or decreasing the victim recovery rate does not affect the active scammer population within the system.
    \item Scammers who choose to quit on their own have a greater impact on reducing overall scam activity than removing them through enforcement actions.
    \item The rate of individual susceptibility is a primary determinant of the dynamics of propagation of scams within a given population. Essentially, the influx of newly vulnerable targets directly fueled the expansion of fraudulent activity.
\end{itemize}
A key recommendation from this study is to broaden intervention efforts beyond simply enhancing the scammers' removal rate, also to include dismantling their recruitment pipelines (i.e, by decreasing the rate of $\delta$) and sensitization of the victims (i.e, by reducing the rate of $\psi$). Although the present model serves as a crucial first step in quantifying the interplay between scammers and victims, we acknowledge that it constitutes a simplified representation of a profoundly complex system. Consequently, our future work will be directed towards incorporating more compartmental scenarios and control strategies, including applying optimal control theory to determine the most cost-effective allocation of resources over time.

\section*{Conflict of interest}
\noindent The authors confirm that there are no known financial interests or personal relationships that could be perceived as a potential conflict of interest in relation to the research presented in this article. 

\section*{Authors contribution}
All authors contributed equally to this manuscript. 

\section*{Data and code availability}
The data and code associated with this article can be found on Zenodo \cite{tijani2025canadian} and \href{https://github.com/hamfat/Phishing}{{Phishing GitHub repository}}, respectively.

\end{document}